\documentclass[11pt]{amsart}
\usepackage{amsmath}
\usepackage{amssymb}
\usepackage{amsthm}
\usepackage{latexsym,tikz}
\usepackage{hyperref}
\usepackage{enumerate}
\usepackage[all]{xy}

\newtheorem{thm}{Theorem}[section]

\newtheorem{lem}[thm]{Lemma}
\newtheorem{prop}[thm]{Proposition}
\newtheorem{defn}[thm]{Definition}

\newtheorem{ex}[thm]{Example}

\def\Irr{\mathbf{Irr}}
\def\cG{\mathcal{G}}
\def\I{\mathfrak{I}}
\def\cB{\mathcal{B}}
\def\fB{\mathfrak{B}}
\def\fL{\mathfrak{L}}

\def\cT{\mathcal{T}}

\def\fs{\mathfrak{s}}
\def\fo{\mathfrak{o}}

\def\Gal{\mathrm{Gal}}
\def\Hom{\mathrm{Hom}}
\def\Ind{\mathrm{Ind}}

\def\temp{\mathrm{temp}}
\def\St{\mathrm{St}}

\def\SO{\mathrm{SO}}
\def\SU{\mathrm{SU}}
\def\PGL{\mathrm{PGL}}
\def\PSU{\mathrm{PSU}}
\def\GL{\mathrm{GL}}
\def\Aut{\mathrm{Aut}}
\def\Nrd{\mathrm{Nrd}}
\def\F{\mathbb{F}}

\def\R{\mathbb{R}}
\def\C{\mathbb{C}}
\def\Z{\mathbb{Z}}
\def\T{\mathbb{T}}
\def\Q{\mathbb{Q}}
\def\SL{\mathrm{SL}}
\def\PSL{\mathrm{PSL}}
\def\im{\mathrm{im}}

\def\Prim{\mathbf{Prim}}
\def\W{\mathbf{W}}
\def\i{\mathbf{i}}
\def\j{\mathbf{j}}
\def\k{\mathbf{k}}
\def\q{/\!/}

\begin{document}

\title[$L$-packets and depth]{$L$-packets and depth for $\SL_2(K)$  with $K$ a local function field of characteristic $2$}

\author[S. Mendes]{Sergio Mendes}
\address{ISCTE - Lisbon University Institute\\    Av. das For\c{c}as Armadas\\     1649-026, Lisbon\\   Portugal}
\email{sergio.mendes@iscte.pt} 
\author[R. Plymen]{Roger Plymen}
\address{School of Mathematics, Southampton University, Southampton SO17 1BJ,  England 
\emph{and} School of Mathematics, Manchester University, Manchester M13 9PL, England}
\email{r.j.plymen@soton.ac.uk \quad plymen@manchester.ac.uk}


\date{\today}
\maketitle

\begin{abstract}  Let $\cG = \SL_2(K)$ with $K$ a local function field of characteristic $2$.   We review  Artin-Schreier theory for the
field $K$,  and show that this leads to a parametrization of certain $L$-packets in the smooth dual of $\cG$.  We relate this to a recent geometric conjecture.
The $L$-packets in the principal series are parametrized by quadratic extensions, and the supercuspidal $L$-packets of cardinality $4$ are parametrised by biquadratic extensions. Each supercuspidal packet of cardinality $4$ is accompanied by a singleton packet for $\SL_1(D)$.   We compute the depths of the irreducible constituents of all these $L$-packets for $\SL_2(K)$ and its inner form
$\SL_1(D)$.   

\end{abstract}

\section{Introduction} The special linear group $\SL_2$ has been a mainstay of representation theory for at least 45 years, see \cite{GGPS}. In that book, the authors show how the unitary irreducible representations of $\SL_2(\R)$ and $\SL_2(\Q_p)$ can be woven together in the context of automorphic forms. This comes about in the following way.
The classical notion of a cusp form $f$ in the upper half plane leads first to the concept of a cusp form on the
adele group of $\GL_2$ over $\Q$, and thence to the idea of an automorphic cuspidal representation $\pi_f$ of the adele group of $\GL_2$.
We recall that the adele group of $\GL_2$ is the restricted product of the local groups $\GL_2(\Q_p)$ where $p$ is a place of $\Q$.  If $p$ is infinite then $\Q_p$ is the real field $\R$; if $p$ is finite then $\Q_p$ is the $p$-adic field.  The unitary representation $\pi_f$ may be expressed as $\otimes \pi_p$ with one local representation for each local group $\GL_2(\Q_p)$.   It is this way that the unitary representation theory of groups such as $\GL_2(\Q_p)$ enters into the modern theory of automorphic forms.
\smallskip

Let $X$ be a smooth projective curve over $\F_q$. Denote by $F$ the field $\F_q(X)$ of rational functions on $X$. For any closed point $x$ of $X$ we denote by $F_x$ the completion of $F$ at $x$ and by $\fo_x$ its ring of integers. If we choose a local coordinate $t_x$ at $x$ (i.e., a rational function on $X$ which vanishes at $x$ to order one), then we obtain isomorphisms $F_x \simeq \F_{q_x} ((t_x))$ and $\fo_x \simeq  \F_{q_x} [[t_x]]$, where $\F_{q_x}$ 	is the residue field of $x$; in general, it is a finite extension of $\F_q$ containing $q_x = q^{deg(x)}$ elements.
Thus, we now have a \emph{local function field} attached to each point of X.
\smallskip

With all this in the background, it seems natural to us to study the representation theory of $\SL_2(K)$ with $K$ a local function field.   The case when $K$ has characteristic $2$ has many special features -- and we focus on this case in this article.   A local function field $K$ of characteristic $2$ is of the form
$K= \F_q((t))$,  the field of Laurent series with
coefficients in $\F_q$, with $q=2^f$. This example is particularly
interesting because there are countably many quadratic extensions of
$\F_q((t))$.
\smallskip

Artin-Schreier theory is a branch of Galois theory, and more specifically is a positive characteristic analogue of Kummer theory, for Galois extensions of degree equal to the characteristic $p$.  Artin and Schreier (1927) introduced Artin-Schreier theory for extensions of prime degree $p$, and Witt (1936) generalized it to extensions of prime power degree $p^n$.
If K is a field of characteristic $p$, a prime number, any polynomial of the form
\[X^p - X + \alpha
\]
for $\alpha \in K$, is called an Artin-Schreier polynomial. When $\alpha$ does not lie in the subset  $\{ y \in K \, | \, y=x^p-x \; \mbox{for } x \in K \}$, this polynomial is irreducible in $K[X]$, and its splitting field over $K$ is a cyclic extension of $K$ of degree p. This follows since for any root $\beta$, the numbers
$\beta  + i$, for $1\le i\le p$, form all the root -- by Fermat's little theorem -- so the splitting field is  $K(\beta)$.
Conversely, any Galois extension of $K$ of degree $p$ equal to the characteristic of $K$ is the splitting field of an Artin-Schreier polynomial. This can be proved using additive counterparts of the methods involved in Kummer theory, such as Hilbert's theorem $90$ and additive Galois cohomology. These extensions are called Artin-Schreier extensions.
\smallskip

For the moment, let $F$ be a local nonarchimedean field with odd residual characteristic.    The $L$-packets for $\SL_2(F)$ are classified in the paper 
\cite{LR} by Lansky-Rhaguram.  They comprise: the principal series $L$-packets $\xi_E = \{\pi_E^1, \pi_E^2\}$ where $E/F$ is a quadratic extension; the unramified supercuspidal $L$-packet of cardinality $4$; and the supercuspidal $L$-packets of cardinality $2$.   

We now revert to the case of a local function field $K$ of characteristic $2$.    We consider $\SL_2(K)$.     Drawing on the accounts in \cite{Da,Th1,Th2}, we review  Artin-Schreier theory, adapted to the local function field $K$, with special emphasis on the quadratic extensions of $K$.       

The $L$-packets in the principal series of $\SL_2(K)$ are parametrized by quadratic extensions, and the supercuspidal $L$-packets of cardinality $4$ are parametrised by biquadratic extensions $L/K$. There are countably many such supercuspidal $L$-packets.   In this article, we do not consider supercuspidal $L$-packets of cardinality $2$.   

The concept of \emph{depth} can be traced back to the concept of \emph{level} of a character.  Let $\chi$ be a non-trivial character of $K^{\times}$.   The level of $\chi$ is the least integer $n \geq 0$ such that $\chi$ is trivial on the higher unit group $U_K^{n+1}$, see \cite[p.12]{BH}.  The depth of a Langlands parameter $\phi$ is defined as follows.   Let $r$ be a real number, $r \geq 0$,  let $\Gal (K_s / K)^r$ be the $r$-th ramification subgroup
of the absolute Galois group of $K$. Then the depth of $\phi$ is the smallest number
$d(\phi) \geq 0$ such that $\phi$ is trivial on $\Gal (K_s/K)^r$ for all $r > d(\phi)$.

The \emph{depth} $d(\pi)$ of an irreducible $\cG$-representation $\pi$ was defined by Moy and
Prasad \cite{MoPr1,MoPr2}  in terms of filtrations $P_{x,r} (r \in \R_{\geq 0})$ of
the parahoric subgroups $P_x \subset \cG$.

Let $\cG = \SL_2(K)$.   Let $\Irr(\cG)$ denote the smooth dual of $\cG$.   Thanks to a recent article \cite{ABPS1}, we have,
 for every Langlands parameter $\phi \in \Phi (\cG)$
with L-packet $\Pi_\phi (\cG) \subset \Irr (\cG)$
\begin{equation}\label{D}
d(\phi) = d(\pi) \quad \text{ for all } \pi \in \Pi_\phi (\cG) .
\end{equation}

The equation (\ref{D}) is a big help in the computation of the depth $d(\pi)$.   To each biquadratic extension $L/K$, there is attached a Langlands parameter 
$\phi = \phi_{L/K}$, and an $L$-packet $\Pi_{\phi}$ of cardinality $4$.   The  depth of the parameter $\phi_{L/K}$ depends on the extension $L/K$.   More precisely, the  numbers $d(\phi)$ depend on the breaks in the upper ramification filtration of the Galois group 
\[
\Gal(L/K) = \Z/2\Z \times \Z/2\Z.
\]
For certain extensions $L/K$ the allowed depths can be any odd number $1,3,5,7, \ldots$. For the other extensions $L/K$, the allowed depths are $3,5,7,9, \ldots$.
   Accordingly, the depth of each irreducible supercuspidal representation $\pi$ in the packet $\Pi_{\phi}$ is given by the formula
\begin{align}\label{d}
d(\pi) = 2n+1
\end{align}
where $n = 0,1,2,3,\ldots$ or $1,2,3,4,\ldots$ depending on $L/K$.
Let $D$ be a central division algebra of dimension $4$ over $K$.   The parameter $\phi$ is relevant for the inner form $\SL_1(D)$, which admits singleton $L$-packets, and  the depths are given by the formula (\ref{d}).

  This contrasts with the case of $\SL_2(\Q_p)$ with $p>2$.  Here there is a unique biquadratic extension $L/K$, and a unique tamely ramified parameter $\phi : \Gal(L/K) \to \SO_3(\R)$ of depth zero.

We move on to consider the geometric conjecture in \cite{ABPS}.   Let $\fB(\cG)$ denote the Bernstein spectrum of $\cG$, let $\fs \in \fB(\cG)$, and let 
$T^{\fs}, W^{\fs}$ denote the complex torus, finite group, attached by Bernstein to $\fs$.   For more details at this point, we refer the reader to \cite{R}.   
The Bernstein decomposition  provides us, inter alia,  with the following data: a canonical disjoint union
\[
\Irr(\cG) = \bigsqcup \Irr(\cG)^{\fs}
\]
and, for each $\fs \in \fB(\cG)$, a finite-to-one surjective map
\[
\Irr(\cG)^{\fs} \to T^{\fs}/W^{\fs}
\]
onto the quotient variety $T^\fs/W^\fs$.   
The geometric conjecture in \cite{ABPS} amounts to a refinement of these statements.   The refinement comprises the assertion that  we have a \emph{bijection}
\begin{align}\label{bij}
\Irr(\cG)^s \simeq T^{\fs}\q W^{\fs}
\end{align}
where $T^{\fs}\q W^{\fs}$ is the \emph{extended quotient} of the torus $T^{\fs}$ by the finite group $W^{\fs}$.   If the action of $W^\fs$ on $T^\fs$ is free, then the extended quotient is equal to the ordinary quotient $T^\fs/W^\fs$.  If the action is not free, then the extended quotient is a finite disjoint union of quotient varieties, one of which is the ordinary quotient.   The bijection (\ref{bij}) is subject to certain constraints, itemised in \cite{ABPS}.

In the case of $\SL_2$, the torus $T^{\fs}$ is of dimension $1$, and the finite group $W^{\fs}$ is either $1$ or $\Z/2\Z$.   So, in this context, the content of the conjecture is rather modest: but a proof is required, and such a proof is duly given in \S 7.

We thank Anne-Marie Aubert for her careful reading of the manuscript, and for alerting us to the article  \cite{LR}.   
Thanks also to Chandan Dalawat for a valuable exchange of emails and for the references \cite{Da2,Da}.

\section{Artin-Schreier theory}

Let $K$ be a local field with positive characteristic $p$, containing the $n$-th roots of unity $\zeta_n$. The cyclic
extensions of $K$ whose degree $n$ is coprime with $p$ are described
by Kummer theory. It is well known that any cyclic
extension $L/K$ of degree $n$, $(n,p)=1$, is generated by a root
$\alpha$ of an irreducible polynomial $x^n-a\in K[x]$. We fix an algebraic closure $\overline{K}$ of $K$ and a separable closure
$K^s$ of $K$ in $\overline{K}$.
If $\alpha\in K^s$ is a root of $x^n-a$ then $K(\alpha)/K$ is a
cyclic extension of degree $n$ and is called a Kummer extension of
$K$.

Artin-Schreier theory aims to describe cyclic extensions of degree equal to
or divisible by $ch(K)=p$. It is therefore an analogue of Kummer theory, where the
role of the polynomial $x^n-a$ is played by $x^n-x-a$. Essentially,
every cyclic extension of $K$ with degree $p=ch(K)$ is generated by
a root $\alpha$ of $x^p-x-a\in K[x]$.

 Let $\wp$ denote the Artin-Schreier endomorphism of the additive group
$K^s$ \cite{Ne}:
\[
\wp : K^s \to K^s, \quad x \mapsto x^p-x.
\]

Given $a\in K$ denote by $K(\wp^{-1}(a))$ the extension $K(\alpha)$,
where $\wp(\alpha)=a$ and $\alpha\in K^s$. We have the following characterization of finite cyclic Artin-Schreier extensions of degree $p$:

\begin{thm}
\begin{itemize}
\item[$(i)$] Given $a\in K$, either $\wp(x)-a\in K[x]$ has one root in $K$ in which case it has all the $p$ roots are in $K$, or is irreducible.
\item[$(ii)$] If $\wp(x)-a\in K[x]$ is irreducible then $K(\wp^{-1}(a))/K$ is a cyclic extension of degree $p$, with $\wp^{-1}(a)\subset K^s$.
\item[$(iii)$] If $L/K$ be a finite cyclic extension of degree $p$, then $L=K(\wp^{-1}(a))$, for some $a\in K$.
\end{itemize}
\end{thm}
(See \cite[p.34]{Th1} for more details.)

\bigskip

We fix now some notation. $K$ is a local field with characteristic $p>1$ with finite
residue field $k$. The field of constants $k=\F_q$ is a finite extension of $\F_p$, with degree $[k:\F_p]=f$ and $q=p^f$.

Let $\mathfrak{o}$ be the ring of integers in $K$ and denote by $\mathfrak{p}\subset\mathfrak{o}$ the (unique) maximal ideal of $\mathfrak{o}$.
This ideal is principal and any generator of $\mathfrak{p}$ is called a uniformizer. A choice of uniformizer $\varpi\in\mathfrak{o}$ determines isomorphisms $K\cong\F_q((\varpi))$, $\mathfrak{o}\cong\F_q[[\varpi]]$ and $\mathfrak{p}=\varpi\mathfrak{o}\cong\varpi\F_q[[\varpi]]$.

A normalized valuation on $K$ will be denoted by $\nu$, so that $\nu(\varpi)=1$ and $\nu(K)=\Z$. The group of units is denoted by $\mathfrak{o}^{\times}$.

\subsection{The Artin-Schreier symbol}


Let $L/K$ be a finite Galois extension. Let $N_{L/K}$ be the norm map and denote by $\Gal(L/K)^{ab}$ the abelianization of $\Gal(L/K)$. The reciprocity map is a group isomorphism
\begin{equation}
K^{\times}/N_{L/K} L^{\times}\stackrel{\simeq}\longrightarrow \Gal(L/K)^{ab}.
\end{equation}\label{Artin symbol}
The Artin symbol is obtained by composing the reciprocity map with the canonical morphism $K^{\times}\rightarrow K^{\times}/N_{L/K} L^{\times}$
\begin{equation}
b\in K^{\times} \mapsto(b,L/K)\in \Gal(L/K)^{ab}.
\end{equation}

From the Artin symbol we obtain a pairing
\begin{equation}\label{A-S pairing KxK*}
K\times K^{\times}\longrightarrow\mathbb{Z}/p\mathbb{Z} , (a,b)\mapsto(b,L/K)(\alpha)-\alpha,
\end{equation}
where $\wp(\alpha)=a$, $\alpha\in K^s$ and $L=K(\alpha)$.

\begin{defn}
Given $a\in K$ and $b\in K^{\times}$, the Artin-Schreier symbol is defined by
\[
[a,b)=(b,L/K)(\alpha)-\alpha.
\]
\end{defn}

The Artin-Schreier symbol is a bilinear map satisfying the following properties, see \cite[p.341]{Ne}:

\begin{equation}
 [a_1+a_2,b)=[a_1,b)+[a_2,b);
\end{equation}
\begin{equation}
 [a,b_1b_2)=[a,b_1)+[a,b_2);
\end{equation}
\begin{equation}\label{Artin-Schreier symbol (iii)}
 [a,b)=0, \forall a\in K \Leftrightarrow b\in N_{L/K}L^{\times}, L=K(\alpha) \textrm{ and }\wp(\alpha)=a;
\end{equation}
\begin{equation}\label{Artin-Schreier symbol (iv)}
 [a,b)=0, \forall b\in K^{\times} \Leftrightarrow a\in\wp(K).
\end{equation}

\subsection{The groups $K/\wp(K)$ and $K^{\times}/K^{\times p}$}

In this section we recall some properties of the groups $K/\wp(K)$ and $K^{\times}/K^{\times p}$ and use them to redefine the pairing (\ref{A-S pairing KxK*}). Dalawat \cite{Da2,Da} interprets $K/\wp(K)$ and $K^{\times}/K^{\times p}$ as $\F_p$-spaces. This interpretation will be particularly useful in $\S 4$.

Consider the additive group $K$. By \cite[Proposition 11]{Da}, the $\F_p$-space $K/\wp(K)$ is countably infinite . Hence, $K/\wp(K)$ is infinite as a group.
\begin{prop}\label{group K/PK}
$K/\wp(K)$ is a discrete abelian torsion group.
\end{prop}
\begin{proof}
The ring of integers decomposes as a (direct) sum
\[
\mathfrak{o}=\mathbb{F}_q+\mathfrak{p}
\]
and we have
\[
\wp(\mathfrak{o})=\wp(\mathbb{F}_q)+\wp(\mathfrak{p}).
\]
The restriction $\wp:\mathfrak{p}\rightarrow\mathfrak{p}$ is an isomorphism, see \cite[Lemma 8]{Da}. Hence,
\[
\wp(\mathfrak{o})=\wp(\mathbb{F}_q)+\mathfrak{p}
\]
and $\mathfrak{p}\subset\wp(K)$. It follows that $\wp(K)$ is an open subgroup of $K$ and $K/\wp(K)$ is discrete. Since $\wp(K)$ is annihilated by $p$, $K/\wp(K)$ is a torsion group.
\end{proof}

Now we concentrate on the multiplicative group $K^{\times}$. For any $n>0$, let $U_n$ be the kernel of the reduction map from $\mathfrak{o}^{\times}$ to $(\mathfrak{o}/\mathfrak{p}^n)^{\times}$. In particular, $U_1=ker(\mathfrak{o}^{\times}\rightarrow k^{\times})$. The $U_n$ are $\Z_p$-modules , because they are commutative pro-p-groups. By \cite[Proposition 20]{Da2}, the $\Z_p$-module $U_1$ is not finitely generated. As a consequence, $K^{\times}/K^{\times p}$ is infinite, see \cite[Corollary 21]{Da2}. The next result gives a characterization of the topological group $K^{\times}/K^{\times p}$.

\begin{prop}\label{K*/K^*p is profinite}
$K^{\times}/K^{\times p}$ is a profinite abelian $p$-torsion group.
\end{prop}
\begin{proof}
There is a canonical isomorphism $K^{\times}\cong\mathbb{Z}\times \mathfrak{o}^{\times}$. The group of units is a direct product $\mathfrak{o}^{\times}\cong \F_q^{\times}\times U_1$, with $q=p^f$. By \cite[p.25]{Iw}, the group $U_1$ is a direct product of countable many copies of the ring of $p$-adic integers
$$U_1\cong\mathbb{Z}_p\times\mathbb{Z}_p\times\mathbb{Z}_p\times...=\prod_{\mathbb{N}}\mathbb{Z}_p.$$
Give $\mathbb{Z}$ the discrete topology and $\mathbb{Z}_p$ the $p$-adic topology. Then, for the product topology, $K^{\times}=\mathbb{Z}\times\mathbb{Z}/(q-1)\mathbb{Z}\times\prod_{\mathbb{N}}\mathbb{Z}_p$ is a topological group, locally compact, Hausdorff and totally disconnected.

Now, $K^{\times p}$ decomposes as a product of countable many components
\[
K^{\times p}\cong p\mathbb{Z}\times\mathbb{Z}/(q-1)\mathbb{Z}\times p\mathbb{Z}_p\times p\mathbb{Z}_p\times...\]
\[\hskip -0.45cm =p\mathbb{Z}\times\mathbb{Z}/(q-1)\mathbb{Z}\times \prod_{\mathbb{N}}p\mathbb{Z}_p.
\]
Note that $p\mathbb{Z}/(q-1)\mathbb{Z}=\mathbb{Z}/(q-1)\mathbb{Z}$, since $p$ and $q-1$ are coprime. Denote by $z=\prod_n z_n$ an element of $\prod_{\mathbb{N}}\mathbb{Z}_p$, where $z_n=\sum_{i=0}^{\infty}a_{i,n}p^i \in\mathbb{Z}_p$, for every $n$.

The map
\[
\varphi:\mathbb{Z}\times\mathbb{Z}/(q-1)\mathbb{Z}\times\prod_{\mathbb{N}}\mathbb{Z}_p\rightarrow\mathbb{Z}/p\mathbb{Z}\times\prod_{\mathbb{N}}\mathbb{Z}/p\mathbb{Z}
\]
defined by
\[
(x,y,z)\mapsto (x\!\!\mod p),\prod_n pr_0(z_n))
\]
where $pr_0(z_n)=a_{0,n}$ is the projection, is clearly a group homomorphism.

Now, $\mathbb{Z}/p\mathbb{Z}\times\prod_{\mathbb{N}}\mathbb{Z}/p\mathbb{Z}=\prod_{n=0}^{\infty}\mathbb{Z}/p\mathbb{Z}$ is a topological group for the product topology, where each component $\mathbb{Z}/p\mathbb{Z}$ has the discrete topology.  It is compact, Hausdorff and totally disconnected. Therefore, $\prod_{n=0}^{\infty}\mathbb{Z}/p\mathbb{Z}$ is a profinite group.

Since
\[
ker\varphi=p\mathbb{Z}\times\mathbb{Z}/(q-1)\mathbb{Z}\times\prod_{\mathbb{N}}p\mathbb{Z}_p,
\]
it follows that there is an isomorphism of topological groups
\[
K^{\times}/K^{\times p}\cong\prod_{\mathbb{N}_0}\mathbb{Z}/p\mathbb{Z},
\]
where $K^{\times}/K^{\times p}$ is given the quotient topology. Therefore, $K^{\times}/K^{\times p}$ is profinite.
\end{proof}

From propositions \ref{group K/PK} and \ref{K*/K^*p is profinite}, $K/\wp(K)$ is a discrete abelian group and $K/K^{\times p}$ is an abelian profinite group, both annihilated by $p=ch(K)$. Therefore, Pontryagin duality coincides with $Hom(-,\mathbb{Z}/p\mathbb{Z})$ on both of these groups, see \cite{Th2}.  The pairing (\ref{A-S pairing KxK*}) restricts to a pairing

\begin{equation}\label{A.-S- Pairing}
[.,.): K/\wp(K)\times K^{\times}/K^{\times p}\rightarrow\Z/p\Z.
\end{equation}
which we refer from now on to the \textbf{Artin-Schreier pairing}. It follows from (\ref{Artin-Schreier symbol (iii)}) and (\ref{Artin-Schreier symbol (iv)}), that the pairing is nondegenerate (see also \cite[Proposition 3.1]{Th2}). The next result shows that the pairing is perfect.

\begin{prop}\label{A-S symbol induce quadratic character}
The Artin-Schreier symbol induces isomorphisms of topological groups
\[
K^{\times}/K^{\times p}\stackrel{\simeq}\longrightarrow \Hom(K/\wp(K),\mathbb{Z}/p\mathbb{Z}), bK^{\times p}\mapsto(a+\wp(K)\mapsto[a,b))
\]
and
\[
K/\wp(K)\stackrel{\simeq}\longrightarrow \Hom(K^{\times}/K^{\times p},\mathbb{Z}/p\mathbb{Z}), a+\wp(K)\mapsto(bK^{\times p}\mapsto[a,b))
\]
\end{prop}
\begin{proof}
The result follows by taking $n=1$ in Proposition $5.1$ of \cite{Th2}, and from the fact that Pontryagin duality for the groups $K/\wp(K)$ and $K^{\times}/K^{\times p}$ coincide with $Hom(-,\mathbb{Z}/p\mathbb{Z})$ duality. Hence, there is an isomorphism of topological groups between each such group and its bidual.
\end{proof}

Let $B$ be a subgroup of the additive group of $K$ with finite index such that $\wp(K)\subseteq B\subseteq K$. The composite of two finite abelian Galois extensions of exponent $p$ is again a finite abelian Galois extension of exponent $p$. Therefore, the composite
\[
K_B=K(\wp^{-1}(B))=\prod_{a\in B}K(\wp^{-1}(a))
\]
is a finite abelian Galois extension of exponent $p$. On the other hand, if $L/K$ is a finite abelian Galois extension of exponent $p$, then $L=K_B$ for some subgroup $\wp(K)\subseteq B\subseteq K$ with finite index.

All such extensions lie in the maximal abelian extension of exponent $p$, which we denote by $K_p=K(\wp^{-1}(K))$. The extension $K_p/K$ is infinite and Galois. The corresponding Galois group $G_p=Gal(K_p/K)$ is an infinite profinite group and may be identified, under class field theory, with $K^{\times}/K^{\times p}$, see \cite[Proposition 5.1]{Th2}. The case $ch(K)=2$ leads to $G_2\cong K^{\times}/K^{\times 2}$ and will play a fundamental role in the sequel.

\section{Quadratic characters}   From now on we take $K$ to be a local function field with $ch(K)=2$.
Therefore, $K$ is of the form $\mathbb{F}_q((\varpi))$ with $q = 2^f$.


When $K = \mathbb{F}_q((\varpi))$, we have,  according to \cite[p.25]{Iw},
\[
U_1\cong \mathbb{Z}_2\times\mathbb{Z}_2\times\mathbb{Z}_2\times...=\prod_{\mathbb{N}}\mathbb{Z}_2
\]
with countably infinite many copies of $\mathbb{Z}_2$, the ring of $2$-adic integers.

Artin-Schreier theory provides a way to parametrize all the
quadratic extensions of $K=\mathbb{F}_q((\varpi))$. By proposition \ref{K*/K^*p is profinite}, there is a
bijection between the set of quadratic extensions of
$\mathbb{F}_q((\varpi))$ and the group
\[
\mathbb{F}_q((\varpi))^{\times}/\mathbb{F}_q((\varpi))^{\times 2}\cong\prod_{\mathbb{N}_0}\mathbb{Z}/2\mathbb{Z}= G_2
\]
where $G_2$ is the Galois group of the \emph{maximal abelian extension of
exponent} $2$. Since $G_2$ is an infinite profinite group, there are
countably many quadratic extensions.

To each quadratic extension $K(\alpha)/K$, with $\alpha^2-\alpha=a$, we associate the Artin-Schreier symbol
\[
[a,.) : K^{\times}/K^{\times 2}\rightarrow\mathbb{Z}/2\mathbb{Z}.
\]
Now, let $\varphi$ denote the isomorphism $\mathbb{Z}/2\mathbb{Z}\cong\mu_2(\mathbb{C})=\{\pm 1\}$ with the group of roots of unity. We obtain, by composing with the Artin-Schreier symbol, a unique multiplicative quadratic character
\begin{equation}\label{def A-S character}
\chi_a: K^{\times}\rightarrow\mathbb{C}^{\times}, \,\, \chi_a=\varphi([a,.))
\end{equation}

Proposition \ref{A-S symbol induce quadratic character} shows that every quadratic character of $\mathbb{F}_q((\varpi))^{\times}$ arises in this way.

\begin{ex}\label{example: unr. quadratic ext.}
The unramified quadratic extension of $K$ is $K(\wp^{-1}(\mathfrak{o}))$, see \cite{Da} proposition $12$. According to Dalawat, the group $K/\wp(K)$ may be regarded as an $\mathbb{F}_2$-space and the image of $\mathfrak{o}$ under the canonical surjection $K\rightarrow K/\wp(K)$ is an $\mathbb{F}_2$-line, i.e., isomorphic to $\mathbb{F}_2$. Since $\wp_{|\mathfrak{p}}:\mathfrak{p}\rightarrow\mathfrak{p}$ is an isomorphism, the image of $\mathfrak{p}$ in $K/\wp(K)$ is $\{0\}$, see lemma $8$ in \cite{Da}. Now, choose any $a_0\in\mathfrak{o}\backslash\mathfrak{p}$ such that the image of $a_0$ in $\mathfrak{o}/\mathfrak{p}$ has nonzero trace in $\F_2$, see \cite[Proposition 9]{Da}. The quadratic character $\chi_{a_0}=\varphi([a_0,.))$ associated with $K(\wp^{-1}(\mathfrak{o}))$ via class field theory is precisely the unramified character $(n\mapsto(-1)^n)$ from above. Note that any other choice $b_0\in\mathfrak{o}\backslash\mathfrak{p}$, with $a_0\neq b_0$, gives the same unique unramified character, since there is only one nontrivial coset $a_0+\wp(K)$ for $a_0\in\mathfrak{o}\backslash\mathfrak{p}$.
\end{ex}

\smallskip

Let $\cG$ denote  $\SL_2(K)$, let $\cB$ be the standard Borel subgroup of $\cG$, let $\cT$ be the diagonal subgroup of $\cG$.
Let $\chi$ be a character of $\cT$. Then, $\chi$ inflates to a character of $\cB$. Denote by $\pi(\chi)$ the (unitarily) induced representation $\Ind_{\cB}^{\cG}(\chi)$. The representation space $V(\chi)$ of $\pi(\chi)$ consists of locally constant complex valued functions $f:\cG\rightarrow\mathbb{C}$ such that, for every $a\in K^{\times}$, $b\in K$ and $g\in \cG$, we have

\[
f\bigg(\left( \begin{array}{cc}
 a & b \\
 0 & a^{-1}
\end{array} \right)\Bigg)=|a|\chi(a)f(g)
\]

The action of $\cG$ on $V(\chi)$ is by right translation. The representations $(\pi(\chi),V(\chi))$ are called (unitary) principal series of $\cG=SL_2(K)$.

Let $\chi$ be a quadratic character of $K^{\times}$. The reducibility of the induced representation $\Ind_B^G(\chi)$ is well known in zero characteristic. Casselman proved that the same result holds in characteristic $2$ and any other positive characteristic $p$.

\begin{thm}\cite{Ca,Ca2}\label{Casselman's th.}
The representation $\pi(\chi)= \Ind_{\cB}^{\cG}(\chi)$ is reducible if, and only if, $\chi$ is either $|.|^{\pm}$ or a nontrivial quadratic character of $K^{\times}$.
\end{thm}

For a proof see \cite[Theorems 1.7, 1.9]{Ca} and \cite[\S 9]{Ca2}.

\medskip

From now on, $\chi$ will be a quadratic character. It is a classical result that the unitary
principal series for $\GL_2$ are irreducible. For a representation of $\GL_2$ parabolically induced by
$1\otimes\chi$, Clifford theory tells us that the dimension of the intertwining algebra of its restriction to
$\SL_2$ is $2$. This is exactly the induced representation of $\SL_2$ by $\chi$:
\[
\Ind_{\widetilde{B}}^{\GL_2(K)}(1\otimes\chi)_{|\SL(2,K)}\stackrel{\simeq}\longrightarrow \Ind_{B}^{\SL_2(K)}(\chi)
\]
where $\widetilde{B}$ denotes the standard Borel subgroup of $\GL_2(K)$. This leads to reducibility of the induced representation
$\Ind_B^G(\chi)$ into two inequivalent constituents.  Thanks to M. Tadic for helpful comments at this point.

The two irreducible constituents
\begin{align}
\pi(\chi)= \Ind_{B}^{G}(\chi)=\pi(\chi)^+\oplus\pi_(\chi)^-
\end{align}
define an $L$-packet $\{\pi(\chi)^+ , \pi(\chi)^-\}$ for $\SL_2$.

\section{Biquadratic extensions of $\F_q((\varpi))$}

Quadratic extensions $L/K$ are obtained by adjoining an $\mathbb{F}_2$-line $D\subset K/\wp(K)$. Therefore, $L=K(\wp^{-1}(D))=K(\alpha)$ where $D=span\{a+\wp(K)\}$, with $\alpha^2-\alpha=a$. In particular, if $a_0\in\mathfrak{o}\backslash\mathfrak{p}$ such that the image of $a_0$ in $\mathfrak{o}/\mathfrak{p}$ has nonzero trace in $\F_2$, the $\mathbb{F}_2$-line $V_0=span\{a_0+\wp(K)\}$ contains all the cosets $a_i+\wp(K)$ where $a_i$ is an integer and so $K(\wp^{-1}(\mathfrak{o}))=K(\wp^{-1}(V_0))=K(\alpha_0)$ where $\alpha_0^2-\alpha_0=a_0$ gives the unramified quadratic extension.

Biquadratic extensions are computed the same way, by considering $\mathbb{F}_2$-planes $W=span\{a+\wp(K), b+\wp(K)\}\subset K/\wp(K)$. Therefore, if $a+\wp(K)$ and $b+\wp(K)$ are $\mathbb{F}_2$-linearly independent then $K(\wp^{-1}(W)):=K(\alpha, \beta)$ is biquadratic, where $\alpha^2-\alpha=a$ and $\beta^2-\beta=b$, $\alpha, \beta\in K^s$. Therefore, $K(\alpha, \beta)/K$ is biquadratic if $b-a\not\in\wp(K)$.

A biquadratic extension containing the line $V_0$ is of the form $K(\alpha_0,\beta)/K$. There are countably many quadratic extensions $L_0/K$ containing the unramified quadratic extension. They have ramification index $e(L_0/K)=2$. And there are countably many biquadratic extensions $L/K$ which do not contain the unramified quadratic extension. They have ramification index $e(L/K)=4$.

So, there is a plentiful supply of biquadratic extensions $K(\alpha, \beta)/K$.

\subsection{Ramification}

The space $K/\wp(K)$ comes with a filtration

\begin{equation}\label{Filtration K/P(K)}
0\subset_1 V_0\subset_f V_1=V_2\subset_f V_3=V_4\subset_f ...\subset K/\wp(K)
\end{equation}
where $V_0$ is the image of $\mathfrak{o}_K$  and $V_i$ ($i>0$) is the image of $\mathfrak{p}^{-i}$ under the canonical surjection $K\rightarrow K/\wp(K)$. For $K=\mathbb{F}_q((\varpi))$ and $i>0$, each inclusion $V_{2i}\subset_f V_{2i+1}$ is a sub-$\mathbb{F}_2$-space of codimension $f$. The $\F_2$-dimension of $V_n$ is
\begin{equation}\label{F_2 dim.}
dim_{\F_2}V_n=1+\lceil n/2 \rceil f,
\end{equation}
for every $n\in\mathbb{N}$, where $\lceil x \rceil$ is the smallest integer bigger than $x$.

\bigskip

Let $L/K$ denote a Galois extension with Galois group $G$. For each $i\geq -1$ we define the $i^{th}$-ramification subgroup of $G$ (in the lower numbering) to be:
$$G_i=\{\sigma\in G: \sigma(x)-x \in\mathfrak{p}_L^{i+1}, \forall x\in\mathfrak{o}_L\}.$$
An integer $t$ is a \emph{break} for the filtration $\{G_i\}_{i\geq -1}$ if $G_t\neq G_{t+1}$. The study of ramification groups $\{G_i\}_{i\geq -1}$ is equivalent to the study of breaks of the filtration.

There is another decreasing filtration with upper numbering $\{G^i\}_{i\geq -1}$ and defined by the \emph{Hasse-Herbrand function} $\psi=\psi_{L/K}$:
$$G^u=G_{\psi(u)}.$$
In particular, $G^{-1}=G_{-1}=G$ and $G^0=G_0$, since $\psi(0)=0$.

\bigskip

Let $G_2=Gal(K_2/K)$ be the Galois group of the maximal abelian extension of exponent $2$, $K_2=K(\wp^{-1}(K))$. Since $G_2\cong K^{\times}/K^{\times 2}$ (proposition \ref{K*/K^*p is profinite}), the pairing $K^{\times}/K^{\times 2}\times K/\wp(K)\rightarrow \Z/2\Z$ from (\ref{A.-S- Pairing}) coincides with the pairing $G_2\times K/\wp(K)\rightarrow \Z/2\Z$.

The profinite group $G_2$ comes equipped with a ramification filtration $(G_2^u)_{u\geq -1}$ in the upper numbering, see \cite[p.409]{Da}. For $u\geq 0$, we have an orthogonal relation \cite[Proposition 17]{Da}
\begin{equation}\label{orthogonal}
(G_2^u)^{\bot}=\overline{\mathfrak{p}^{-\lceil u \rceil+1}}=V_{\lceil u \rceil-1}
\end{equation}
under the pairing $G_2\times K/\wp(K)\rightarrow \Z/2\Z$.

\bigskip
Since the upper filtration is more suitable for quotients, we will compute the upper breaks. By using the Hasse-Herbrand function it is then possible to compute the lower breaks in order to obtain the lower ramification filtration.

According to \cite[Proposition $17$]{Da}, the positive breaks in the filtration $(G^v)_v$ occur precisely at integers prime to $p$. So, for $ch(K)=2$, the positive breaks will occur at odd integers. The lower numbering breaks are also integers. If $G$ is cyclic of prime order, then there is a unique break for any decreasing filtration $(G^v)_v$ (see \cite{Da}, Proposition $14$). In general, the number of breaks depends on the possible filtration of the Galois group.

Given a plane $W\subset K/\wp(K)$, the filtration (\ref{Filtration K/P(K)}) $(V_i)_i$ on $K/\wp(K)$ induces a filtration $(W_i)_i$ on $W$, where $W_i=W\cap V_i$. There are three possibilities for the filtration breaks on a plane and we will consider each case individually.

\bigskip

\textbf{Case 1 :} $W$ contains the line $V_0$, i.e. $L_0=K(\wp^{-1}(W))$ contains the unramified quadratic extension $K(\wp^{-1}(V_0))=K(\alpha_0)$ of $K$. The extension has residue degree $f(L_0/K)=2$ and ramification index $e(L_0/K)=2$. In this case, there is an integer $t>0$, necessarily odd, such that the filtration $(W_i)_i$ looks like
$$0\subset_1 W_0=W_{t-1}\subset_1 W_{t}=W.$$

By the orthogonality relation (\ref{orthogonal}), the upper ramification filtration on $G=Gal(L_0/K)$ looks like
$$\{1\}=...=G^{t+1}\subset_1G^{t}=...=G^0\subset_1G^{-1}=G$$
Therefore, the upper ramification breaks occur at $-1$ and $t$.


\medskip

The number of such $W$ is equal to the number of planes in $V_t$ containing the line $V_0$ but but not contained in the subspace $V_{t-1}$. This number can be computed and equals the number of biquadratic extensions of $K$ containing the unramified quadratic extensions and with a pair of upper ramification breaks $(-1,t)$, $t>0$ and odd. Here is an example.

\begin{ex}
The number of biquadratic extensions containing the unramified quadratic extension and with a pair of upper ramification breaks $(-1,1)$ is equal to the number of planes in an $1+f$-dimensional $\F_2$-space, containing the line $V_0$. There are precisely
$$1+2+2^2+...+2^{f-1}=\frac{1-2^f}{1-2}=q-1$$
of such biquadratic extensions.
\end{ex}

\bigskip

\textbf{Case 2.1 :} $W$ does not contains the line $V_0$ and the induced filtration on the plane $W$ looks like
$$0=W_{t-1}\subset_2 W_{t}=W$$
for some integer $t$, necessarily odd.

The number of such $W$ is equal to the  number of planes in $V_t$ whose intersection with $V_{t-1}$ is $\{0\}$. Note that, there are no such planes when $f=1$. So, for $K=\F_2((\varpi))$, \textbf{case 2.1} does not occur.

Suppose $f>1$. By the orthogonality relation, the upper ramification ramification filtration on $G=Gal(L/K)$ looks like
$$\{1\}=...=G^{t+1}\subset_2G^{t}=...=G^{-1}=G$$
Therefore, there is a single upper ramification break occurring at $t>0$ and is necessarily odd.


\medskip

For $f=1$ there is no such biquadratic extension. For $f>1$, the number of these biquadratic extensions  equals the number of planes $W$ contained in an $\F_2$-space of dimension $1+fi$, $t=2i-1$, which are transverse to a given codimension-$f$ $\F_2$-space.

\bigskip

\textbf{Case 2.2 :} $W$ does not contains the line $V_0$ and the induced filtration on the plane $W$ looks like
$$0=W_{t_1-1}\subset_1 W_{t_1}=W_{t_2-1}\subset_1 W_{t_2}=W$$
for some integers $t_1$ and $t_2$, necessarily odd, with $0<t_1<t_2$.

The orthogonality relation for this case implies that the upper ramification filtration on $G=Gal(L/K)$ looks like
$$\{1\}=...=G^{t_2+1}\subset_1G^{t_2}=...=G^{t_1+1}\subset_1G^{t_1}=...=G$$
The upper ramification breaks occur at odd integers $t_1$ and $t_2$.




There is only a finite number of such biquadratic extensions, for a given pair of upper breaks
$(t_1,t_2)$.

\section{Langlands parameter}  We have the following canonical homomorphism:
\[
\W_K \to \W_K^{ab} \simeq K^{\times} \to K^{\times}/K^{\times 2}.
\]
According to \S 2, we also have
\[
K^{\times}/K^{\times 2} \simeq \prod \Z/2\Z
\]
the product over countably many copies of $\Z/2\Z$.  Using the countable axiom of choice, we choose two copies of $\Z/2\Z$. This creates a homomorphism
\[
 \W_K \to \Z/2\Z \times \Z/2\Z.
\]
There are countably many such homomorphisms.

  Following \cite{We}, denote by $\alpha, \beta, \gamma$ the images in $\PSL_2(\C)$ of the elements
\[
z_{\alpha} = \left(
\begin{array}{cc}
i & 0\\
0 & -i
\end{array}
\right),
\quad z_{\beta} = \left(
\begin{array}{cc}
0 & 1\\
-1 & 0
\end{array}
\right),
\quad
z_{\gamma} = \left(
\begin{array}{cc}
0 & i\\
i & 0
\end{array}
\right),
\]
in $\SL_2(\C)$.

Note that $z_\alpha, z_\beta, z_\gamma \in \SU_2(\C)$ so that
\[
\alpha, \beta, \gamma \in \PSU_2(\C) = \SO_3(\R).
\]
  Denote by $J$ the group generated by $\alpha, \beta, \gamma$:
\[
J: = \{\epsilon, \alpha, \beta,\gamma\} \simeq \Z/2\Z \times \Z/2\Z.
\]
The group $J$ is unique up to conjugacy in $G = \PSL_2(\C)$.

The pre-image of $J$ in $\SL_2(\C)$ is the group
$\{\pm 1, \pm z_{\alpha}, \pm z_{\beta}, \pm z_{\gamma} \}$ and is isomorphic to the group $U_8$ of unit quaternions $\{\pm 1, \pm \i, \pm \j, \pm \k\}$.

The centralizer and normalizer of $J$ are given by
\[
C_G(J) = J, \quad 	N_G(J) = O
\]
 where $O \simeq S_4$ the symmetric group on $4$ letters.
 The quotient  $O/J \simeq  \GL_2(\Z/2)$ is the full automorphism group of $J$.

Each biquadratic extension $L/K$ determines a Langlands parameter
\begin{align}\label{phi}
\phi : \Gal(L/K) \to \SO_3(\R) \subset \SO_3(\C)
\end{align}

Define
\begin{align}\label{S}
S_{\phi} & = C_{\PSL_2(\C)}(\im \, \phi)
\end{align}
Then we  have $S_{\phi}	= J$, since $C_G(J) = J$,  and whose conjugacy class depends only on $L$, since $O/J = \Aut(J)$.

  Define the new group
\begin{align*}
\mathcal{S}_{\phi} & = C_{\SL_2(\C)}(\im \, \phi)
\end{align*}
To align with the notation in \cite{ABPS2}, replace $\phi^\sharp$ in \cite{ABPS2} by $\phi$ in the present article.
We have the short exact sequence
\[
1 \to \mathcal{Z}_{\phi} \to \mathcal{S}_{\phi} \to S_{\phi} \to 1
\]
with $\mathcal{Z}_{\phi} = \Z/2\Z$.

Let $D$ be a central division algebra of dimension $4$ over $K$, and let $\Nrd$ denote the reduced norm on $D^\times$.  Define
\[
\SL_1(D) = \{x \in D^\times : \Nrd (x) = 1\}.
\]
Then $\SL_1(D)$ is an inner form of $\SL_2(K)$.
In the local Langlands correspondence \cite{ABPS2} for the inner forms of $\SL_2$,
the L-parameter $\phi$ is enhanced by elements $\rho \in \Irr(\mathcal{S}_{\phi})$.   Now the group $\mathcal{S}_{\phi} \simeq U_8$ admits four characters $\rho_1, \rho_2, \rho_3, \rho_4$ and one irreducible representation $\rho_0$ of degree $2$.

The parameter $\phi$ creates a big packet with five elements, which are allocated to $\SL_2(K)$ or $\SL_1(D)$ according to central characters.   So $\phi$ assigns an $L$-packet $\Pi_{\phi}$ to $\SL_2(K)$ with $4$ elements, and a singleton packet to the inner form $\SL_1(D)$.  None of these packets contains the Steinberg representation of $\SL_2(K)$ and so each $\Pi_{\varphi}$ is a  supercuspidal $L$-packet with $4$ elements.

 To be explicit:
$\phi$ assigns to $\SL_2(K)$  the supercuspidal packet
\[
\{\pi(\phi, \rho_1), \pi(\phi, \rho_2), \pi(\phi, \rho_3), \pi(\phi, \rho_4)\}
\]
and to $\SL_1(D)$ the singleton packet
\[
\{\pi(\phi, \rho_0)\}
\]
and this phenomenon occurs countably many times.

Each supercuspidal packet of four elements is the \emph{JL-transfer} of the singleton packet, in the following sense:  the irreducible supercuspidal representation $\theta$ of $\GL_2(K)$ which yields the $4$-packet upon restriction to $\SL_2(K)$ is the image in the JL-correspondence of the irreducible smooth representation $\psi$ of $\GL_1(D)$ which yields two copies of $\pi(\phi, \rho_0)$ upon restriction to $\SL_1(D)$:
\[
\theta = JL(\psi).
\]

Each parameter $\phi : \W_K \to \PGL_2(\C)$ lifts to a Galois representation
\[
\phi : \W_K \to \GL_2(\C).
\]
   This representation is \emph{triply imprimitive}, as in \cite{We}.   Let $\mathfrak{T}(\phi)$ be the group of characters $\chi$ of $\W_K$ such that
   $\chi \otimes \phi \simeq \phi$.   Then  $\mathfrak{T}(\phi)$ is non-cyclic of order $4$.

   \section{Depth}  Let $L/K$ be a biquadratic extension. We fix an algebraic closure $\overline{K}$ of $K$ such that $L\subset \overline{K}$. From the inclusion $L\subset \overline{K}$, there is a natural surjection

$$\pi_{L/K} : \Gal(\overline{K}/K)\rightarrow \Gal(L/K)$$

Let $K^{ur}$ be the maximal unramified extension of $K$ in $\overline{K}$ and let $K^{ab}$ be the maximal abelian extension of $K$ in $\overline{K}$.
We have a commutative diagram, where the horizontal maps are the canonical maps and the vertical maps are the natural projections

\[
\xymatrix{
1 \ar[r]& I_{\overline{K}/K} \ar[d]_{\alpha_1}\ar[r]^{\iota_1}
& Gal(\overline{K}/K) \ar[d]_{\pi_1}\ar[r]^{p_1} & Gal(K^{ur}/K) \ar[d]_{id}\ar[r]&  1\\
1 \ar[r]& I_{K^{ab}/K} \ar[d]_{\alpha_2}\ar[r]^{\iota_2}
& Gal(K^{ab}/K) \ar[d]_{\pi_2}\ar[r]^{p_2} & Gal(K^{ur}/K) \ar[d]_{\beta}\ar[r]& 1\\
1 \ar[r]& \I_{L/K} \ar[r]^{\iota_3}
& Gal(L/K) \ar[r]^{p_3} & Gal(L\cap K^{ur}/K) \ar[r] & 1
}
\]

In the above notation, we have $\pi_{L/K}=\pi_2\circ\pi_1$.

Let
\begin{equation}
... \I^{(2)}\subset \I^{(1)} \subset \I^{(0)}\subset G=\Gal(L/K)
\end{equation}
be the filtration of the
relative inertia subgroup $\I^{(0)}=\I_{L/K}$ of $\Gal(L/K)$, $\I^{(1)}$ is the wild inertia subgroup, and so on... Note that $\I^{(r)}$
is the restriction of the filtration $G^r$ of $G=Gal(L/K)$ to the subgroup $\I_{L/K}$, i.e, $\I^{(r)}=\iota_3(G^r)$.

Let
\begin{equation}
... I^{(2)}\subset I^{(1)} \subset I^{(0)}\subset G=\Gal(\overline{K}/K)
\end{equation}
be the filtration of the absolute inertia subgroup $I^{(0)}=I_{\overline{K}/K}$ of $\Gal(K^s/K)$, $I^{(1)}$ is the wild inertia subgroup, and so on...

\medskip

\begin{lem} We have
\[
(\forall r ) \; \pi_{L/K} I^{(r)} = \I^{(r)}
\]

\end{lem}

\begin{proof} This follows immediately from the above diagram. Here, we identify $I^{(r)}$ with $\iota_1(I^{(r)})$ and $\I^{(r)}$ with $\iota_3(\I^{(r)})$.

\end{proof}

\begin{lem}\label{ddd}  Let $L/K$ be a biquadratic extension, let $\phi$ be the Langlands parameter (\ref{phi}), $\phi = \alpha \circ \pi_{L/K}$ with $\alpha : \Gal(L/K) \to \SO_3(\R)$.     Then we have
 $d(\phi) = r - 1$ where $r$ is the least integer for which $\I^{(r)} = 1$.
\end{lem}

\begin{proof} The depth of a Langlands parameter $\phi$  is easy to define.
For $r \in \R \geq 0$ let $\Gal(F_s/F)^r$ be the $r$-th ramification subgroup of the absolute Galois group of $F$.
Then the depth of $\phi$  is the smallest number $d(\phi) \geq 0$ such that $\phi$  is trivial on $\Gal(F_s/F)^r$  for all $r > d(\phi)$.

 Note that $\alpha$ is \emph{injective}.  Therefore
\[
\phi(I^{(r)}) = 1 \iff (\alpha \circ \pi_{L/K})I^{(r)}) = 1 \iff \alpha(\I^{(r)}) = 1 \iff \I^{(r)} = 1.
\]
\end{proof}

For example, the parameter $\phi$ has depth zero if it is \emph{tamely ramified}, i.e. the least integer $r$ for which $\I^{(r)} = 1$ is $r = 1$.   The relative wild inertia group is $1$, but the relative inertia group is not $1$.

\medskip

\textbf{Case $1$:} There are two ramification breaks occurring at $-1$ and some odd integer $t>0$:

$$\{1\}=...= \I^{(t+1)}\subset \I^{(t)}=... \I^{(0)}= \I_{L/K}\subset \Gal(L/K), \quad d(\varphi) = t  $$

The allowed depths are $1,3,5,7, \ldots$.

\medskip

\textbf{Case $2.1$:} One single ramification break occurs at some odd integer $t>0$::

\[
\{1\}=...= \I^{(t+1)}\subset \I^{(t)}=...= \I^{(0)}= \I_{L/K}= \Gal(L/K); \quad d(\varphi)  = t
\]

The allowed depths are $1,3,5,7,\ldots$.

\medskip

\textbf{Case $2.2$:} There are two ramification breaks occurring at some odd integers $t_1<t_2$

\[
\{1\}=...= \I^{(t_2+1)}\subset \I^{(t_2)}=...= \I^{(t_1+1)}\subset \I^{(t_1)}=...= \I^{(0)}= \I_{L/K}= \Gal(L/K); \quad d(\varphi) = t_2
\]

The allowed depths are $3,5,7,9,\ldots$.

(In the above, $\I^{(0)}= \I_{L/K}$)

\begin{thm} Let $L/K$ be a biquadratic extension, let $\phi$ be the Langlands parameter (\ref{phi}).   For every $\pi \in
\Pi_{\phi}(\SL_2(K))$ and $\pi \in \Pi_{\phi}(\SL_1(D))$  there is an equality of depths:
\[
d(\pi) = d(\phi).
\]
The depth of each element in the $L$-packet $\Pi_{\phi}$ 
is given by the largest break in the ramification of the
Galois group $\Gal(L/K)$.   
The allowed depths are $1,3,5,7, \ldots$ except in Case 2.2, when the allowed depths are $3,5,7, \ldots$.   
\end{thm}

\begin{proof}  This follows from Lemma (\ref{ddd}), the above computations, and Theorem 3.4 in \cite{ABPS1}.\end{proof}

  This contrasts with the case of $\SL_2(\Q_p)$ with $p>2$.  Here there is a unique biquadratic extension $L/K$, and a unique tamely ramified parameter $\phi : \Gal(L/K) \to \SO_3(\R)$ of depth zero.

 \subsection{Quadratic extensions}  Let $E/K$ be a quadratic extension. There are two kinds: the unramified one $E_0=K(\alpha_0)$ and countably many totally (and wildly) ramified $E =K(\alpha)$.
 
 \begin{thm}   For the unramified principal series $L$-packet $\{\pi_E^1, \pi_E^2\}$, we have
 \[
 d(\pi_E^1) = d(\pi_E^2) = -1.
 \]
 For the ramified principal series $L$-packet $\{\pi_E^1, \pi_E^2\}$, we have 
  \[
 d(\pi_E^1) = d(\pi_E^2) = n
 \]
 with $n = 1,2,3,4, \ldots$.    
 \end{thm}
 \begin{proof}  Case 1: $E_0/K$ unramified. Then, $f(E_0/K)=2$. In this case, we have $G_0=\{1\}$, and $G_0=G^0=\mathfrak{I}_{E_0/K}$. There is only one ramification break at $t=0$ and the filtration of $G= \Gal(E_0/K)$ in the upper numbering is
$$\{1\}=G^0\subset G^{-1}=G=\mathbb{Z}/2\mathbb{Z}.$$
The filtration on the relative inertia $\mathfrak{I}^{(t)}$ is
$$\{1\}=\mathfrak{I}_{L_0/K}\subset G=\mathbb{Z}/2\mathbb{Z}$$
with only one break at $t=0$.  Negative depth, as expected.

Case 2: $E/K$ is totally ramified. Then, $e(E/K)=2$, which is divisible by the residue degree, so the extension is wildly ramified. In this case, there is one break at some $t \geq 1$. This is because of wild ramification, since $G^1=\{1\}$ if and only if the extension is tamely ramified. The filtration of $G$ in the upper numbering is
$$\{1\}=G^{t+1}\subset G^t=...=G^0=G=\mathbb{Z}/2\mathbb{Z}$$

The filtration on the relative inertia $\mathfrak{I}^{(r)}$ is
$$\{1\}=\mathfrak{I}^{(t+1)}\subset \mathfrak{I}^{(t)}=...=G=\mathbb{Z}/2\mathbb{Z}$$
with only one break at $t \geq1$.
\end{proof}

    \section{A commutative triangle}

In this section we confirm part of the  geometric conjecture in \cite{ABPS} for $\SL_2(\mathbb{F}_q((\varpi)))$. We begin by recalling the underlying ideas of the conjecture.

Let $\cG$ be the group of $K$-points of a connected reductive group over a nonarchimedean local field $K$.   The Bernstein decomposition  provides us, inner alia,  with the following data: a canonical disjoint union
\[
\Irr(\cG) = \bigsqcup \Irr(\cG)^{\fs}
\]
and, for each $\fs \in \fB(\cG)$, a finite-to-one surjective map
\[
\Irr(\cG)^{\fs} \to T^{\fs}/W^{\fs}
\]

The geometric conjecture in \cite{ABPS} amounts to a refinement of these statements.   The refinement comprises the assertion that  we have a \emph{bijection}
\[
\Irr(\cG)^s \simeq (T^{\fs}\q W^{\fs})_2
\]
where $(T^{\fs}\q W^{\fs})_2$ is the \emph{extended quotient of the second kind} of the torus $T^{\fs}$ by the finite group $W^{\fs}$.   This bijection is subject to certain constraints, itemised in \cite{ABPS}.

We proceed to define the extended quotient of the second kind. Let $W$ be a finite group and let $X$ be a complex affine algebraic variety. Suppose that $W$ is acting on $X$ as automorphisms of $X$. Define
\[
\widetilde{X}_2:  =  \{(x,\tau) : \tau \in \Irr(W_x)\}.
\]
Then $W$ acts on $\widetilde{X}_2$:
\[
\alpha(x, \tau)=(\alpha \cdot x,\alpha_* \tau).
\]
\begin{defn}
The extended quotient of the second kind is defined as
\[
(X\q W)_2 :=\widetilde{X}_2/W.
\]
\end{defn}
Thus the extended quotient of the second kind is the ordinary quotient for the action of $W$ on $\widetilde{X}_2$.

We recall that $(G,T)$ are the complex dual groups of $(\cG,\cT)$, so that $G = \PSL_2(\C)$.      Let $\W_K$ denote the Weil group of $K$.
  If $\phi$ is an $L$-parameter
  \[
  \W_K \times \SL_2(\C) \to G
  \]
  then an \emph{enhanced Langlands parameter} is a pair $(\phi, \rho)$ where $\phi$ is a parameter and $\rho \in \Irr(S_{\phi})$.

\begin{thm} \label{Triangle}
Let $\cG = \SL_2(K)$ with $K = \F_q((\varpi)))$.   Let $\fs = [\cT,\chi]_G$ be a point in the Bernstein spectrum for the principal series of
$\cG$.    Let $\Irr(\cG)^{\fs}$ be the corresponding  Bernstein component in $\Irr(\cG)$. Then 
there is a commutative triangle of natural bijections
\[
\xymatrix{
& (T^{\fs}/\!/W^\fs)_2 \ar[dr]\ar[dl] & \\
\Irr(\mathcal{G})^\fs    \ar[rr] & & \fL(G)^\fs
}
\]
where  $\fL(G)^\fs$ denotes the equivalence classes of enhanced parameters attached to $\fs$.
\end{thm}

\begin{proof}    We recall that $T^\fs = \{\psi\chi : \psi \in \Psi(\cT)\}$ where $\Psi(\cT)$ is the group of all unramified quasicharacters of $\cT$.  
With $\lambda \in T^\fs$, we define the parameter $\phi(\lambda)$ as follows:
\[
\phi(\lambda): W_K \times \SL_2(\C)  \to \PSL_2(\C), \quad (w \Phi_K^n, Y)  \mapsto \left( \begin{array}{cc}\lambda(\varpi)^n & 0\\
0 & 1 \end{array} \right)_*
\]
where $A_*$ is the image in $\PSL_2(\C)$ of $A \in \SL_2(\C)$, $Y \in \SL_2(\C)$, $w \in I_K$ the inertia group, and $\Phi_K$ is a geometric Frobenius. Define, as in \S 3,
\[
\pi(\lambda): = \Ind_\cB^\cG (\lambda).
\]

\textbf{Case 1}.   $\lambda^2 \neq 1$.   
Send the pair $(\lambda, 1) \in T^{\fs} \q W^{\fs}$ to $\pi(\lambda) \in \Irr(\cG)^\fs$ (via the left slanted arrow) and to $\phi(\lambda) \in \fL(G)^\fs$ ( via the right slanted arrow).

\textbf{Case 2}.  Let $\lambda^2 = 1, \lambda \neq 1$. 
Let $\phi = \phi(\lambda)$.   
To compute $S_{\phi}$, let $1,w$ be representatives of the Weyl group $W = W(G)$.   Then we have
\[
C_G(\im\,\phi) = T \sqcup wT
\]
So $\phi$ is a non-discrete parameter, and we have
\[
S_{\phi} \simeq \Z/2\Z.
\]

We have two enhanced parameters, namely $(\phi, 1)$ and $(\phi, \epsilon)$ where $\epsilon$ is the nontrivial character of $\Z/2\Z$.



Since $\lambda^2 = 1$, there is a point of reducibility.   We send \[
(\lambda, 1) \mapsto \pi(\lambda)^+, \; (\lambda, \epsilon) \mapsto \pi(\lambda)^-
\]
 via the left slanted arrow, and \[
 (\lambda, 1) \mapsto (\phi(\lambda),1), \; (\lambda, \epsilon) \mapsto (\phi(\lambda), \epsilon)
 \]
  via the right slanted arrow.  Note that this \emph{includes} the case when $\lambda$ is the unramified quadratic character of $K^{\times}$.





\textbf{Case 3}.   Let $\lambda  = 1$.
The \emph{principal parameter}
\[
\phi_0 : \W_K \times  \SL_2(\C)  \to \SL_2(\C) \to \PSL(2,\C).
\]
is a discrete parameter for which $S_{\phi_0} = 1$.
In the local Langlands correspondence for $\cG$, the enhanced parameter $(\phi_0, 1)$ corresponds to the Steinberg representation $\St$ of $\SL_2(K)$.
Note also that, when $\phi = \phi(1)$, we have $S_{\phi} = 1$.   
  We send
  \[
  (1,1) \mapsto \pi(1), \quad (1,\epsilon) \mapsto \St
   \]
   via the left slanted arrow
   and
   \[
   (1,1) \mapsto (\phi(1),1), \quad (1,\epsilon) \mapsto (\phi_0,1)
   \]
   via the right slanted arrow.  This establishes that the geometric conjecture in \cite{ABPS} is valid for $\Irr(\cG)^{\fs}$. 
   \end{proof}

Let $L/K$ be a quadratic extension of $K$.   Let $\lambda$ be the quadratic character which is trivial on $N_{L/K}L^\times$. Then $\lambda$ factors through $\Gal(L/K) \simeq K^\times/ N_{L/K}L^\times \simeq \Z/2\Z$ and $\phi(\lambda)$ factors through $\Gal(L/K) \times \SL_2(\C)$.  The parameters $\phi(\lambda)$ serve as  parameters for the
$L$-packets in the principal series of $\SL_2(K)$.

It follows from \S3 that, when $K = \F_q((\varpi))$,  there are countably many $L$-packets in the principal series of $\SL_2(K)$.

\subsection{The tempered dual}
If we insist, in the definition of $T^\fs$,  that the unramified character $\psi$ shall be unitary, then we obtain a copy $\T^\fs$ of the circle $\T$.
We then obtain a compact version of the commutative triangle, in which the tempered dual $\Irr^\temp(\cG)^\fs$ determined by $\fs$ occurs on the left, and the bounded enhanced parameters $\fL^b(G)^\fs$ determined by $\fs$ occur on the right.
We now isolate the bijective map
\begin{align}\label{map}
(\T^{\fs}/\!/W^\fs)_2 \to
\Irr^\temp(\mathcal{G})^\fs  
\end{align}
and restrict ourselves to the case where $\T^\fs$ contains two \emph{ramified} quadratic characters.
Let  $\T: = \{ z \in \C : |z| = 1\}, \; W: = \Z/2\Z$.       We then have $T^\fs = \T, \; W^\fs = W$  and the generator of $W$ acts on  $\T$ sending $z$ to $z^{-1}$.  

The left-hand-side and the right-hand-side of the map (\ref{map}) each has its own natural topology, as we proceed to explain.

The topology on $(\T\q W)_2$ comes about as follows.  Let
\[
\Prim ( C(\T) \rtimes W)
\]
  denote the primitive ideal space of the noncommutative
$C^*$-algebra $C(\T) \rtimes W$.   By the classical Mackey theory for semidirect products, we have a canonical bijection
\begin{align}\label{JJJ}
\Prim ( C(\T) \rtimes W) \simeq (\T\q W)_2.
\end{align}
The primitive ideal space on the left-hand side of (\ref{JJJ}) admits the Jacobson topology.   So the right-hand side of (\ref{JJJ}) acquires, by transport of structure, a compact non-Hausdorff topology.  The  following picture is intended to portray this topology.  

\bigskip

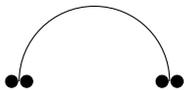
\begin{tikzpicture}\label{pic}
\hskip 4.0cm
\draw (3,0) arc (0:180:1cm);
\draw (0.9,0) circle (0.08cm);
\fill[black] (0.9,0) circle (0.08cm);
\draw (1.1,0) circle (0.08cm);
\fill[black] (1.1,0) circle (0.08cm);
\draw (2.9,0) circle (0.08cm);
\fill[black] (2.9,0) circle (0.08cm);
\draw (3.1,0) circle (0.08cm);
\fill[black] (3.1,0) circle (0.08cm);
\end{tikzpicture}

\bigskip

The reduced $C^*$-algebra of $\cG$  is liminal, and its primitive ideal space is in canonical bijection with the tempered dual of $\cG$. Transporting the Jacobson topology on the primitive ideal space, we obtain a locally compact topology on the tempered dual of $\cG$, see \cite[3.1.1, 4.4.1, 18.3.2]{Dix}.   This makes 
$\Irr^{\temp}(\cG)^\fs$ into a compact space, in the induced topology.

We conjecture that these two topologies make (\ref{map}) into a homeomorphism. This is a strengthening of the geometric conjecture \cite{ABPS}.  In that case, the double-points in the picture arise precisely when the corresponding (parabolically) induced representation has two irreducible constituents.   This conjecture is true for $\SL_2(\Q_p)$ with $p > 2$, see \cite[Lemma 1]{P}. While in conjectural mode, we mention the following point: the standard Borel subgroup of $\SL_2(K)$ admits countably many ramified quadratic characters and so, following the construction  in \cite{ChP}, the geometric conjecture predicts that tetrahedra of reducibility will occur countably many times; however, the $R$-group machinery is not, to our knowledge, available in positive characteristic, so this remains conjectural.

\end{document}